\documentclass[12pt]{amsart}
\usepackage{epsfig} % sollte vor german stehen
\usepackage{color}
\usepackage{psfrag}
\usepackage{graphicx}
\usepackage{amssymb}
\usepackage{amsmath}
\usepackage{latexsym}
\usepackage{enumerate}
\usepackage[mathscr]{eucal}
\usepackage[isolatin]{inputenc}
\usepackage{upref}

\makeatletter
\@namedef{subjclassname@2010}{%
 \textup{2010} Mathematics Subject Classification}
\makeatother

\newtheorem{thm}{Theorem}[section]

\newtheorem{lem}[thm]{Lemma}

\theoremstyle{definition}
\newtheorem{defin}[thm]{Definition}
\newtheorem{rem}[thm]{Remark}
\newtheorem{exa}[thm]{Example}

\numberwithin{equation}{section}

 \newcommand{\setN}{\mathbb{N}}

\begin{document}

\baselineskip=17pt

\title[]
{On a new theory of models\\ for formal mathematical systems}

\author{Matthias Kunik}
\address{Universit\"{a}t Magdeburg\\
IAN \\
Geb\"{a}ude 02 \\
Universit\"{a}tsplatz 2 \\
D-39106 Magdeburg \\
Germany}
\email{matthias.kunik@ovgu.de}

\date{\today}
\maketitle

\begin{abstract}
We study a new model theory for formal mathe\-ma\-tical systems that we developed in a previous paper. We introduce isomorphic and homomorphic structures for formal languages, present some 
results and examples and conclude our paper with a discussion about the reduced set theory RST adapted to our new theory.
\end{abstract}

homomorphisms, new model theory.\\


\section{Introduction}\label{intro}
In \cite{Ku} we have presented a unified theory for formal systems
inclu\-ding recursive systems closely related to formal grammars, including 
the predi\-cate calculus as well as a formal induction principle.
In \cite[(3.13) Induction Rule (e)]{Ku} we use an own rule of inference for formal induction and special predicate symbols for the underlying recursive systems. The remaining rules of inference \cite[(3.13) (a)-(d)]{Ku} without induction and recursive systems provide an own
Hilbert-style predicate calculus which is similar
to that of Shoenfield in his textbook \cite{Shoenfield}.
We will shortly summarize the necessary definitions 
in Section \ref{FMS}.

In \cite{Ku4} we have presented a new definition of a model for a formal mathematical system and derived some basic results,
namely the L\"owenheim-Skolem Theorem for elementary submodels and related results like the correctness of the logic calculus.
We have seen in \cite{Ku4} that it is sufficient to present
the theory of models only for formal mathematical systems without 
using the Induction Rule from \cite[(3.13)(e)]{Ku}.
In Section \ref{structures} we introduce homomorphic and isomorphic structures of formal languages, where formal languages are defined
as formal mathe\-matical systems without basis axioms.
Our theory is more general and flexible for applications as compared to the concepts in \cite{H}, and it is also better suited to a formal approach.
This is supported by the general syntax we use in our formal systems. 
We show how the new concepts fit into the existing theory and illustrate this with some examples.
We conclude our study with a discussion of the reduced axiomatic set theory RST from \cite{Kuset}, which we developed for our model theory.

\section{Formal mathematical systems}\label{FMS}
For the formal mathematical systems from \cite[Sections 3,4]{Ku}
we use five rules of inference, namely rules (a)-(e) from 
\cite[(3.13)]{Ku}. Rule (e) enables formal induction with respect to the recursively enumerable relations generated by an underlying recursive system $S$.
For the theory of models we put $S=S_{\emptyset}=[\,[\,];[\,];[\,]\,]$
in order to avoid the use of rule (e). 
Then we can shortly write $M=[A; P; B]$
instead of $M=[S_{\emptyset}; A; P ;B]$
for our formal systems. 
We denote the countably infinite alphabet of variables by
$X=\{\,{\bf x_1}\,,\,{\bf x_2}\,,\,{\bf x_3}\,,\, \ldots \,\}$.

\begin{defin} {\bf Formal mathematical systems}\\ 
Let $M = [A; P; B]$ be a formal mathematical system and ${\mathcal L}$
a given subset of $A$-lists from \cite[Section 3]{Ku} with the properties 
\begin{itemize}
\item[(i)] $X \subseteq {\mathcal L} $\,,
\item[(ii)] $\lambda \frac{\mu}{x} \in {\mathcal L}$ ~
for all $\lambda, \mu \in {\mathcal L}$, $x \in X$\,,
\item[(iii)] all formulas in $B$ contain only argument lists in ${\mathcal L}$\,.
\end{itemize}
Then $[M;{\mathcal L}]$ is also called a formal
mathematical system (with restricted argument lists ${\mathcal L}$).
A formula in $[M;{\mathcal L}]$ is a formula in $M$ 
which has only argument lists in ${\mathcal L}$\,.
A proof $[\Lambda]$ in $[M;{\mathcal L}]$ is a proof in $M$,
using rules (a)-(d) with the restrictions 
\begin{itemize}
\item[(iv)] the formulas in $[\Lambda]$ and 
the formulas $F$ and $G$ in \cite[(3.13)(a)-(d)]{Ku}
contain only argument lists in ${\mathcal L}$\,,
\item[(v)] there holds $\lambda \in {\mathcal L}$ 
for the list $\lambda$ in \cite[(3.13)(c)]{Ku}\,.
\end{itemize}
Let $\Pi(M;\mathcal{L})$ be the set of formulas provable
in $[M;\mathcal{L}]$ by using only the rules of inference (a),(b),(c),(d)
from \cite[(3.13)]{Ku}.
\end{defin}
\begin{defin}\label{henkin_def}
Let $[M;\mathcal{L}]$ be a formal mathematical system.
\begin{itemize}
\item[(a)] $[M;\mathcal{L}]$ is called \textit{contradictory} if 
$F \in \Pi(M;\mathcal{L})$ for every formula $F$ in $[M;\mathcal{L}]$,
otherwise we say that $[M;\mathcal{L}]$ is \textit{consistent}.
\item[(b)] $[M;\mathcal{L}]$ is called a \textit{Henkin system}
if for every variable $x \in X$ and for every formula $F$ in $[M;\mathcal{L}]$ 
with $\mbox{free}(F) \subseteq \{x\}$ there is a constant $c \in \mathcal{L}$
such that
$\begin{displaystyle}
\to \exists x F \, F\frac{c}{x} \,\in \, \Pi(M;\mathcal{L})\,.
\end{displaystyle}$
We obtain from \cite[(3.19)]{Ku} that the latter condition may be replaced with
\begin{align*}
\leftrightarrow \exists x F \, F\frac{c}{x} \,\in \, \Pi(M;\mathcal{L})\,.
\end{align*}
\item[(c)] $[M;\mathcal{L}]$ is called \textit{complete} if 
$$F \notin \Pi(M;\mathcal{L}) \Leftrightarrow \neg F \in \Pi(M;\mathcal{L}) $$
for every closed formula $F$ in $[M;\mathcal{L}]$\,.
\end{itemize}
\end{defin}
\begin{defin}\label{extensions}
Given are two formal mathematical systems $[M;\mathcal{L}]$ and 
$[M';\mathcal{L}']$ with $M=[A; P; B]$ and $M'=[A'; P'; B']$.
\begin{itemize}
\item[(a)] We say that $[M';\mathcal{L}']$ is an \textit{extension} 
of $[M;\mathcal{L}]$ if
$$
A \subseteq A'\,, \quad 
P \subseteq P'\,, \quad 
\mathcal{L} \subseteq \mathcal{L}' ~\textrm{~and~} ~
\Pi(M;\mathcal{L}) \subseteq \Pi(M';\mathcal{L}') \,.
$$
\item[(b)] Let $[M';\mathcal{L}']$ be an extension of 
$[M;\mathcal{L}]$. If we have in addition
$$
F \in \Pi(M';\mathcal{L}') \implies F \in \Pi(M;\mathcal{L})
$$
for all formulas $F$ in $[M;\mathcal{L}]$, then 
$[M';\mathcal{L}']$ is called a \textit{conservative extension} of 
$[M;\mathcal{L}]$. 
\end{itemize}
\end{defin}

\section{Homomorphic and isomorphic structures\\ of formal mathematical languages}
\label{structures}
Let $[M;\mathcal{L}]$ be a formal mathematical system  with $M=[A;P;B]$.\\
\textit{A model} $\mathcal{D}$ of $[M;\mathcal{L}]$  
consists of the following ingredients:
\begin{itemize}
\item[(1)] We have a nonempty set $\mathcal{D}_*$, also called the universe of the model. The members $d \in \mathcal{D}_*$ are called the individuals of the universe.
\item[(2)] For each individual $d \in \mathcal{D}_*$ 
we have exactly one name $\alpha_d$ and the set of names
$
\mathcal{N}=\{\alpha_d\,:\,d \in \mathcal{D}_*\,\}\,.
$
It is understood that different individuals have different names and that the names in $\mathcal{N}$ are different from the symbols in $[M;\mathcal{L}]$.
\item[(3)] Put $\hat{A} = A \cup \mathcal{N}$,
$\hat{M}=[\hat{A}; P; B]$ and
\begin{align*} \hat{{\mathcal L}}   = 
\{\, & \lambda \frac{\kappa_1}{x_1}...\frac{\kappa_m}{x_m} \,\,:\,\, 
\lambda \in {\mathcal L},\,
x_1,...,x_m \in X,\,\\ & \kappa_1,...,\kappa_m \in 
\mathcal{N}\,,\, m \geq 0\}\,.
\end{align*}
There results an extended mathematical system $[\hat{M};\hat{\mathcal{L}}]$, see also \cite[Corollary (4.9)]{Ku}.
Let $\hat{\mathcal{L}}_*$ be the set of all lists 
$\lambda \in \hat{\mathcal{L}}$
without variables, i.e. we have $\mbox{var}(\lambda)= \emptyset$.
\item[(4)] We have a \textit{surjective} function 
$\mathcal{D} : \hat{\mathcal{L}}_* \to \mathcal{D}_*$ with
\begin{equation*}
\mathcal{D}\left( \lambda \frac{\mu}{x}\right)=
\mathcal{D}\left( \lambda \frac{\alpha_{\mathcal{D}(\mu)}}{x}\right)
\end{equation*}
for all $\mu \in \hat{\mathcal{L}}_*$, for all variables $x \in X$
and for all $\lambda \in \hat{\mathcal{L}}$ with 
$\mbox{var}(\lambda) \subseteq \{ x \}$.
\item[(5)]To each predicate symbol $p \in P$ 
and each $n \in \setN_0=\{0,1,2,3,\ldots\}$ we assign an $n$-ary predicate $p_n \subseteq \mathcal{D}_*^n$. Especially for $n=0$ we either have a truth value
$p_0=\top$ or $p_0 = \bot$.
\footnote{If there is no risk of confusion, we may also put $\bot=\emptyset$, $\top=\{\emptyset\}=\mathcal{D}_*^0$.} 
\item[(6)] We have an extension of the function $\mathcal{D}$ which assigns a truth value
to each closed formula of $[\hat{M};\hat{\mathcal{L}}]$.
This extension is also denoted by $\mathcal{D}$ and is given by
\begin{itemize}
\item[6.1] $\begin{displaystyle}
\mathcal{D}(\sim \lambda,\mu)=\top \Leftrightarrow
\mathcal{D}(\lambda)=\mathcal{D}(\mu)
\end{displaystyle}$ 
for all $\lambda, \mu \in \hat{\mathcal{L}}_*$,
\item[6.2] $\mathcal{D}(p)=p_0 \in \{\top ,\bot\}$ for all $p \in P$,
$\begin{displaystyle}
\mathcal{D}(p \, \lambda_1,\ldots,\lambda_n)=\top \Leftrightarrow
(\mathcal{D}(\lambda_1),\ldots,\mathcal{D}(\lambda_n)) \in p_n
\end{displaystyle}$ for all $n \in \setN=\{1,2,3,\ldots\}$ and all
$\lambda_1,\ldots,\lambda_n \in \hat{\mathcal{L}}_*$\,.
\item[6.3] We have for all closed formulas $F,G \mbox{~of~} [\hat{M};\hat{\mathcal{L}}]$:\\
$\begin{displaystyle}
\mathcal{D}(\neg F)=\top \Leftrightarrow
\mathcal{D}(F)=\bot\,,
\end{displaystyle}$ \\
$\begin{displaystyle}
\mathcal{D}(\to F G)=\top \Leftrightarrow
(\mathcal{D}(F) \Rightarrow \mathcal{D}(G))\,,
\end{displaystyle}$ \\
$\begin{displaystyle}
\mathcal{D}(\leftrightarrow F G)=\top \Leftrightarrow
(\mathcal{D}(F) \Leftrightarrow \mathcal{D}(G))\,,
\end{displaystyle}$ \\
$\begin{displaystyle}
\mathcal{D}(\& \, F G)=\top \Leftrightarrow
(\mathcal{D}(F) ~\mbox{and}~ \mathcal{D}(G))\,,
\end{displaystyle}$ \\
$\begin{displaystyle}
\mathcal{D}(\vee \, F G)=\top \Leftrightarrow
(\mathcal{D}(F) ~\mbox{or}~ \mathcal{D}(G))\,.
\end{displaystyle}$ 
\item[6.4] We have for all $x \in X$ and all formulas 
$F \mbox{~of~}  [\hat{M};\hat{\mathcal{L}}]$ with $\mbox{free}(F) \subseteq \{x\}$:\\
$\begin{displaystyle}
\mathcal{D}(\forall x \, F)=\top \Leftrightarrow
\left( \mathcal{D}\left(F \frac{\lambda}{x}\right)=\top ~\mbox{for all}~
\lambda \in \hat{\mathcal{L}}_* \right) \,,
\end{displaystyle}$ \\
$\begin{displaystyle}
\mathcal{D}(\exists x \, F)=\top 
\Leftrightarrow
\left(~\mbox{there exists}~\lambda \in \hat{\mathcal{L}}_* 
~\mbox{with}~
\mathcal{D}\left(F \frac{\lambda}{x}\right)=\top
\right)\,.
\end{displaystyle}$ 
\end{itemize}
\item[(7)] Let $F$ be a formula in 
$[\hat{M};\hat{\mathcal{L}}]$ with 
$\mbox{free}(F) = \{x_1,\ldots,x_m\}$, $x_1,\ldots,x_m \in X$ 
and $m \geq 0$. We say that $F$ is valid in $\mathcal{D}$ iff
$\begin{displaystyle}
\mathcal{D}\left(
F \frac{\lambda_1}{x_1}\ldots\frac{\lambda_m}{x_m}
\right)=\top
\end{displaystyle}$
for all $\lambda_1,\ldots,\lambda_m \in \hat{\mathcal{L}}_* $.
Note that this simply means 
$\begin{displaystyle}
\mathcal{D}\left(
\forall x_1 \ldots \forall x_m F 
\right)=\top
\end{displaystyle}$\,.
\item[(8)] We require that every formula $F \in B$ is valid in 
$\mathcal{D}$. Then we say that $\mathcal{D}$ is a model
for $[M;\mathcal{L}]$.
\end{itemize}
\begin{rem} 
(a) We say that an extended mapping $\mathcal{D}$ which satisfies
only Conditions (1)-(7) is a \textit{structure} for the formal mathematical system 
$[M;\mathcal{L}]$. Finally, Condition (8) makes it a \textit{model} for
$[M;\mathcal{L}]$. 

(b) We use a Hilbert-style calculus for our formal mathematical systems. Hence we can use the substitution rule (c) and formulas with free variables in our axioms.
Let $F$ be a formula in $[\hat{\mbox{M}};\hat{\mathcal{L}}]$ and $x \in X$.
Then $F \in \Pi(\hat{\mbox{M}};\hat{\mathcal{L}})$ iff
$\forall x\, F \in \Pi(\hat{\mbox{M}};\hat{\mathcal{L}})$
from \cite[(3.11)(a),(3.13)(b)(d)]{Ku}. This matches well with Condition (7) in our definition of the models.
\end{rem}

\begin{lem}\label{structure_lemma}\cite[Lemma 3.3]{Ku4}
Let $\mathcal{D}$ be a structure for $[M;\mathcal{L}]$.\\ Then we have:
\begin{itemize}
\item[(a)] $d = \mathcal{D}\left( \alpha_d \right)$ for all $d \in \mathcal{D}_*$.
\item[(b)]For $x \in X$ and every formula $H$ in $[\hat{M};\hat{\mathcal{L}}]$
with $\mbox{free}(H) \subseteq \{x\}$ and for all $\mu \in \hat{\mathcal{L}}_*$
we have
\begin{equation*}
\mathcal{D}\left( H \frac{\mu}{x} \right) = 
\mathcal{D}\left( H \frac{\alpha_{\mathcal{D}(\mu)}}{x} \right)\,.
\end{equation*}
\end{itemize}
\end{lem}
Lemma \ref{structure_lemma} is used to prove the correctness of the logical calculus, as formulated in the following result.
\begin{lem}\label{correct}\cite[Lemma 3.5]{Ku4}
Let $\mathcal{D}$ be a model for $[M;\mathcal{L}]$.
We make use of Condition (3) and note that 
$\Pi(\hat{M};\hat{\mathcal{L}})$ denotes the set of all
formulas provable in $[\hat{M};\hat{\mathcal{L}}]$.
Then we obtain for every formula $F$ in 
$[\hat{M};\hat{\mathcal{L}}]$:
%$\begin{displaystyle}
$$
F \in \Pi(\hat{M};\hat{\mathcal{L}}) \Rightarrow F \mbox{~is~valid~in~}
\mathcal{D}\,.
$$
%\end{displaystyle}$
\end{lem}

Let $\mathcal{D}$ be a structure of a formal mathematical system 
$[M;\mathcal{L}]$ with $M=[A;P;B]$. Then we can just drop the basis axioms in $B$:\\

\begin{defin}\label{formal_language}
If $[M;\mathcal{L}]$ is a formal mathematical system without basis axioms, i.e. we have $M=[A;P;\emptyset]$, then we will say that $[M;\mathcal{L}]$ is a \textit{formal mathematical language} or a \textit{formal language} for short.
In this case we will also write $M=[A;P]$ instead of
$M=[A;P;\emptyset]$.
\end{defin}

\begin{rem}\label{rem_language}
Every formal language $[M;\mathcal{L}]$ is also a formal system
with logical axioms given in \cite[(3.9),(3.10),(3.11)]{Ku}, namely the axioms of propositional calculus, the axioms of equality and the quantifier axioms.
Now let $\mathcal{D}$ be a structure for the formal language 
$[M;\mathcal{L}]$. Then $\mathcal{D}$ is also a model for 
$[M;\mathcal{L}]$, and we obtain from Lemma \ref{correct}
that every \textit{provable formula} in $[M;\mathcal{L}]$
is valid in $\mathcal{D}$. Therefore, it is very convenient to identify the original concept of a formal language with 
the specific type of formal systems 
from Definition \ref{formal_language}.
\end{rem}

\begin{exa}\label{shoenfield_structure}
Here we consider a special kind of structures for a formal mathematical language $[M;\mathcal{L}]$ with certain restrictions, which are mainly considered in model theory up to now:

To each constant or function symbol $a \in A$ we assign 
a fixed arity $n \in \setN_0$. 
For $n=0$ we say that $a$ is a constant symbol, and for $n \geq 1$ we say that $a$ is an $n$-ary function symbol.
Then $\mathcal{L}$ consists only on terms which are generated by the following rules.
\begin{itemize}
\item[1.] We have $x \in \mathcal{L}$ 
for all variables $x \in X=\{\,{\bf x_1}\,,\,{\bf x_2}\,,\,{\bf x_3}\,,\, \ldots \,\}$.
\item[2.] We have $a \in \mathcal{L}$ for all constant symbols $a \in A$.
\item[3.] Let $n>0$ and let $f$ be an $n$-ary function symbol in $A$.\\
Then $f(\lambda_1 \ldots \lambda_n) \in \mathcal{L}$ for all terms
$\lambda_1,\ldots,\lambda_n \in \mathcal{L}$.
\end{itemize}
To obtain a structure $\mathcal{D}$ 
for $[M;\mathcal{L}]$ we are following the textbook 
of Shoenfield \cite[Chapter 2.5]{Shoenfield}.

Given is a nonempty set $\mathcal{D}_*$. Its members are called individuals.
By induction with respect to the terms we shall now define an individual
$\mathcal{D}(t)$ for each variable-free term $t$ as follows:
\begin{itemize}
\item[i)]  Each individual $d \in \mathcal{D}_*$ has its own name $\alpha_d$.
Let $\mathcal{N}$ be the set of all these names.
We suppose that the symbols in $[M;\mathcal{L}]$ are different
from the symbols in $\mathcal{N}$ and that 
$d_1=d_2 \Leftrightarrow \alpha_{d_1}=\alpha_{d_2}$
for all $d_1,d_2 \in \mathcal{D}_*$.
Define $[\hat{M};\hat{\mathcal{L}}]$ and $\hat{\mathcal{L}}_*$
as in Condition (3).
\item[ii)] To each constant symbol $a \in \hat{\mathcal{L}}_*$ we assign a value $\mathcal{D}(a) \in \mathcal{D}_*$ 
with $d=\mathcal{D}(\alpha_{d})$ for all $d \in \mathcal{D}_*$.
To each $n$-ary function symbol $f$ we assign an $n$-ary function
$f_{\mathcal{D}} : \mathcal{D}_*^n \to \mathcal{D}_*$.
Let $t_1,\ldots,t_n \in \hat{\mathcal{L}}_*$ be variable-free terms and let
$\mathcal{D}(t_1),\ldots,\mathcal{D}(t_n) \in \mathcal{D}_*$
be defined previously. Then we put
$\begin{displaystyle}
\mathcal{D}(f(t_1\ldots t_n))=
f_{\mathcal{D}}(\mathcal{D}(t_1),\ldots,\mathcal{D}(t_n))
\end{displaystyle}$\,.

\end{itemize}
The remai\-ning steps of the construction are due to Conditions (5)-(7).
Then $\mathcal{D}$ is a structure for $[M;\mathcal{L}]$, and
Shoenfield's first lemma in \cite[Chapter 2.5]{Shoenfield}
is just a variant of our Lemma \ref{structure_lemma}.
\end{exa}

The kind of structures and models we study is more general,
because we don't just model functions and relations.
This is supported by the general syntax we use in our formal systems,
as the following example shows:

\begin{exa}\label{string_example}
We define $A=\{a_1,\ldots,a_n\}$ with $n \geq 1$ 
distinct\footnote{We require $E \cap A = \emptyset$ 
with the alphabet $E$ given in \cite[(3.1)(d)]{Ku}} symbols 
$a_1,\ldots,a_n$ and the set $\mathcal{L}$ of strings 
generated by the following rules:
\begin{itemize}
\item[1.] We have $X \subseteq \mathcal{L}$ with $X \cap A = \emptyset$\,, $X$ the set of variables.
\item[2.] If $\lambda \in \mathcal{L}$ 
and $\mu \in \mathcal{L}$, then
$\lambda \mu \in \mathcal{L}$\,.
\item[3.] We have $a_1,\ldots,a_n \in \mathcal{L}$\,.
\end{itemize}
We obtain a formal language 
$[M;\mathcal{L}]$ with $M = [A;P]$ and a given set $P$ of predicate symbols. Let $\mathcal{L}_*$ be the set of all strings 
$\kappa \in \mathcal{L}$ without variables,
i.e. we have $\mbox{var}(\kappa)=\emptyset$. We put
$\mathcal{D}_*=\mathcal{L}_*$ and define the set of names
$
\mathcal{N} = \{\alpha_{\kappa}\,:\,\kappa \in \mathcal{D}_*\}\,.
$
We may assume without loss of generality that these names are not part of the formal language $[M;\mathcal{L}]$. 
Let $[\hat{M};\hat{\mathcal{L}}]$ result from $[M;\mathcal{L}]$
by adding the names in $\mathcal{N}$ to $[M;\mathcal{L}]$.
Let $\hat{\mathcal{L}}_*$ be the set of all lists 
$\lambda \in \hat{\mathcal{L}}$ without variables.
We prescribe $\lambda \in \hat{\mathcal{L}}_*$. 
Let $\alpha_{\kappa_1},\ldots,\alpha_{\kappa_m}$ with $m \geq 0$ 
and $\kappa_1,\ldots,\kappa_m \in \mathcal{L}_*$
be the complete list of distinct names occurring in $\lambda$,
ordered according to their first occurrence in $\lambda$.
Let $x_1,\ldots x_m$ be distinct variables and let $\tilde{\lambda}$ result from $\lambda$
if we replace $\alpha_{\kappa_1},...,\alpha_{\kappa_m}$ everywhere in $\lambda$ 
by $x_1,...,x_m$, respectively. Since the names in $\mathcal{N}$
are not involved in $[M;\mathcal{L}]$, 
we have $\tilde{\lambda} \in \mathcal{L}$ and can define
$\begin{displaystyle}
\mathcal{D}(\lambda)=
\tilde{\lambda}\frac{\kappa_1}{x_1}\ldots\frac{\kappa_m}{x_m} \in \mathcal{L}_* \end{displaystyle}$ for
$\begin{displaystyle} \lambda=
\tilde{\lambda}\frac{\alpha_{\kappa_1}}{x_1}\ldots
\frac{\alpha_{\kappa_m}}{x_m} \in \hat{\mathcal{L}}_*
\end{displaystyle}$\,.\\
The resulting function $\mathcal{D} : 
\hat{\mathcal{L}}_* \to \mathcal{D}_*$ replaces every name 
$\alpha_{\kappa}$ occurring in $\lambda$ by its original string
$\kappa \in \mathcal{L}_*$ and hence satisfies Condition (4).

Now we can directly follow Conditions (5),(6) and (7) in order to obtain a special structure for $[M;\mathcal{L}]$ with prescribed relations between strings over the 
alphabet $A$. But there results a more general formal language if we omit the third rule for the generation of $\mathcal{L}$,
and structures with strings behave \textit{very differently} with respect to these two languages.
We can also  study the recursively enumerable relations 
in \cite[(1.12)]{Ku}, which is of interest 
because \cite[Theorem 3.5,Theorem 3.6]{Ku2} give us natural models of formal systems including structural induction. 
Then we must also 
take into account equivalence classes of strings.
\end{exa} 

Now we return to the general situation.
In the sequel we consider a formal language 
$[M;\mathcal{L}]$ with $M=[A;P]$.
Let $\mathcal{D}_j$ be a structure for $[M;\mathcal{L}]$ 
with universe $\mathcal{D}_{j,*}$ and the set of names 
$\mathcal{N}_j$ for $j=1,2,3$, respectively. We put
\begin{align*}
\mathcal{N}_1=\{\alpha_{d} : d \in \mathcal{D}_{1,*}\}\,,
\mathcal{N}_2=\{\beta_{d'} : d' \in \mathcal{D}_{2,*}\}\,,
\mathcal{N}_3=\{\gamma_{d''} : d'' \in \mathcal{D}_{3,*}\}\,.
\end{align*}

For each $j \in \{1,2,3\}$ we have to require 
that the names in $\mathcal{N}_j$  are new symbols,
not already present in $[M;\mathcal{L}]$, and that the 
restricted mappings 
$\mathcal{D}_j|_{\mathcal{N}_j} : \mathcal{N}_j \to
\mathcal{D}_{j,*}$ with $\mathcal{D}_1|_{\mathcal{N}_1}(\alpha_{d})=d$, 
$\mathcal{D}_2|_{\mathcal{N}_2}(\beta_{d'})=d'$ and
$\mathcal{D}_3|_{\mathcal{N}_3}(\gamma_{d''})=d''$ 
are bijective, respectively, see Lemma \ref{structure_lemma}(a). 
Apart from these restrictions we can freely choose the names separately for each structure\footnote{According to Shoenfield we use $\alpha_d$, $\beta_{d'}$, $\gamma_{d''}$ like syntactical variables for names}. 
Due to Condition (3) we have to add the names in $\mathcal{N}_j$ of each structure $\mathcal{D}_j$ to $[M;\mathcal{L}]$.
For $\mathcal{D}_j$ with $j=1,2,3$ we form 
$\hat{A_j}=A \cup \mathcal{N}_j$, 
$\hat{M_j}=[\hat{A_j};P]$ and
\begin{align*} \hat{{\mathcal L}_j}   = 
\{\, & \lambda \frac{\kappa_1}{x_1}...\frac{\kappa_m}{x_m} \,\,:\,\, 
\lambda \in {\mathcal L},\,
x_1,...,x_m \in X,\,\\ & \kappa_1,...,\kappa_m \in 
\mathcal{N}_j\,,\, m \geq 0\}\,.
\end{align*}
Since we have not added basis axioms, we obtain an extended mathematical language
$[\hat{M_j};\hat{\mathcal{L}_j}]$ 
for $j=1,2,3$, respectively. Let $\hat{\mathcal{L}}_{j,*}$ be the set of all lists $\lambda \in \hat{\mathcal{L}_j}$ without variables, i.e. we have $\mbox{var}(\lambda)= \emptyset$.

In order to let each structure $\mathcal{D}_j$ of $[M;\mathcal{L}]$ assign a truth value to variable-free prime formulas $p \lambda_1, \cdots, \lambda_n$ of $[\hat{M_j};\hat{\mathcal{L}_j}]$ with 
$p \in P$ we have to prescribe predicates $p_{n,j} \subseteq \mathcal{D}_{j,*}^n$ for each $n \in \setN_0$ with $p_{0,j} \in \{\top ,\bot\}$ 
and $j=1,2,3$, respectively.
We want to find suitable conditions in order to define
an homomorphism or isomorphism 
between the structures $\mathcal{D}_1$, $\mathcal{D}_2$ 
for $[M;\mathcal{L}]$.
First we have a mapping 
$\psi : \mathcal{D}_{1,*} \to \mathcal{D}_{2,*}$.
Next we prescribe $\lambda \in \hat{\mathcal{L}}_{1,*}$ and consider for $m \geq 0$ the complete list of \textit{distinct names} 
$\alpha_{d_1},...,\alpha_{d_m}$ in $\lambda$,
ordered according to their first occurrence in $\lambda$. 
Let $x_1,...,x_m \in X$ be \textit{distinct variables} 
and let $\tilde{\lambda}$ result from $\lambda$
if we replace $\alpha_{d_1},...,\alpha_{d_m}$ everywhere in $\lambda$ 
by $x_1,...,x_m$, respectively. 
Since the names in $\mathcal{N}_1$ are different 
from the symbols in $[M;\mathcal{L}]$, we conclude that 
$\tilde{\lambda} \in \mathcal{L}$ with 
$\mbox{var}(\tilde{\lambda})=\{x_1,...,x_m\}$. Then we put
\begin{align}\label{psidef}
\psi_*(\lambda)=\tilde{\lambda}\,\frac{\beta_{\psi(d_1)}}{x_1}...
\frac{\beta_{\psi(d_m)}}{x_m} \in \hat{\mathcal{L}}_{2,*} \mbox{~~for~~}
\lambda=\tilde{\lambda}\,\frac{\alpha_{d_1}}{x_1}...
\frac{\alpha_{d_m}}{x_m} \in \hat{\mathcal{L}}_{1,*}\,.
\end{align}
In this way we obtain a well defined function 
$\psi_* : \hat{\mathcal{L}}_{1,*} \to \hat{\mathcal{L}}_{2,*}$, 
where $\psi_*$ is uniquely determined from $\psi$. 
Recall that $\mathcal{D}_{1}$ and $\mathcal{D}_{2}$ are structures
of the same formal language $[M;\mathcal{L}]$\,.
\begin{rem}\label{rem_psidef}
We note that formulas \eqref{psidef} for the calculation
of $\psi_*(\lambda)$ also remain valid if we have any list 
of not necessarily distinct variables $x_1,...,x_m$ with 
$\mbox{var}(\tilde{\lambda}) \subseteq \{x_1,...,x_m\}$.
This is because among the variables that occur multiple times,
only the innermost variables in \eqref{psidef} 
occurring in $\tilde{\lambda}$ will contribute to the substitutions.
In \eqref{psidef} we can also relax the condition that the names 
$\alpha_{d_1},...,\alpha_{d_m}$ are distinct, 
because all variables $x_k$ with indices $k$ in the list $x_1,...,x_m$ which are substituted by the same name $\alpha_d=\alpha_{d_k}$, can be replaced with a single new variable in $\tilde{\lambda}$, without changing the final result.
But $\tilde{\lambda} \in \mathcal{L}$ is crucial for \eqref{psidef}.
\end{rem}
\begin{defin}\label{homisodef} 
(a) The mapping 
$\psi : \mathcal{D}_{1,*} \to \mathcal{D}_{2,*}$
is called a \textit{homomorphism} from $\mathcal{D}_{1}$ to 
$\mathcal{D}_{2}$
iff there hold the two conditions
\begin{align*}
\psi\left(\mathcal{D}_{1}(\lambda)\right)=
\mathcal{D}_{2}(\psi_*(\lambda))
\quad \mbox{for~all~} \lambda \in \hat{\mathcal{L}}_{1,*}\,,\\
\big( p_{0,1} \Rightarrow p_{0,2} \big)
\mbox{~and~} 
\big( (d_1,...,d_n) \in p_{n,1} \Rightarrow
 (\psi(d_1),...,\psi(d_n)) \in p_{n,2} \big)
\end{align*}
for all $p \in P$, for all $n \in \setN$ and for all 
$d_1,...,d_n \in \mathcal{D}_{1,*}$\,.\\
\noindent 
(b) If $\psi : \mathcal{D}_{1,*} \to \mathcal{D}_{2,*}$ 
is \textit{bijective} with 
\begin{align*}
\psi\left(\mathcal{D}_{1}(\lambda)\right)=
\mathcal{D}_{2}(\psi_*(\lambda))
\quad \mbox{for~all~} \lambda \in  \hat{\mathcal{L}}_{1,*} \,,\\
\big( p_{0,1} \Leftrightarrow p_{0,2} \big)
\mbox{~and~} 
\big( (d_1,...,d_n) \in p_{n,1} \Leftrightarrow
 (\psi(d_1),...,\psi(d_n)) \in p_{n,2} \big)
\end{align*}
for all $p \in P$, for all $n \in \setN$ and for all 
$d_1,...,d_n \in \mathcal{D}_{1,*}$,
then we say that $\psi$ is an \textit{isomorphism} from
$\mathcal{D}_{1}$ to $\mathcal{D}_{2}$. An isomorphism
$\psi$ from $\mathcal{D}_{1}$ to $\mathcal{D}_{1}$ 
is also called an \textit{automorphism} of $\mathcal{D}_{1}$. 
\end{defin}
\begin{rem}\label{equal_structures}
(a) The two structures $\mathcal{D}_{1}$, $\mathcal{D}_{2}$
in Definition \ref{homisodef} can only
be \textit{equal} if at first 
$\mathcal{D}_{1,*}=\mathcal{D}_{2,*}=\mathcal{D}_{*}$, 
at second $\hat{\mathcal{L}}_{1}=
\hat{\mathcal{L}}_{2}=
\hat{\mathcal{L}}$
with $\mathcal{D}_{2}(\lambda)=\mathcal{D}_{1}(\lambda)$
for all $\lambda \in \hat{\mathcal{L}}$ 
with $\mbox{var}(\lambda)=\emptyset$,
and at third $p_{n,1}=p_{n,2}$ for all $p \in P$, $n \in \setN_0$.
Then the \textit{identity map} 
$\psi = \mbox{id}_{\mathcal{D}_{*}}: 
\mathcal{D}_{*} \to \mathcal{D}_{*}$
is an automorphism of $\mathcal{D}_{1}$.

(b) Here we make use of the notations and assumptions
for the special formal language $[M;\mathcal{L}]$
from Example \ref{shoenfield_structure}.
We consider the homomorphism $\psi$ between the two structures 
$\mathcal{D}_{1}$, $\mathcal{D}_{2}$
for $[M;\mathcal{L}]$ according to Definition \ref{homisodef}(a).
Assume that $f$ is an $n$-ary function symbol, $n \geq 1$, and that
$t_1,\ldots,t_n  \in \hat{\mathcal{L}}_{1,*}$. 
Then $f(t_1 \ldots t_n)\in \hat{\mathcal{L}}_{1,*}$,
and we obtain from \eqref{psidef}
\begin{align}\label{psicompatible}
\psi_*(f(t_1 \ldots t_n))=f(\psi_*(t_1) \ldots \psi_*(t_n)) 
\in \hat{\mathcal{L}}_{2,*}\,.
\end{align}
Let $f_{\mathcal{D}_1} : \mathcal{D}_{1,*}^n \to \mathcal{D}_{1,*}$
and $f_{\mathcal{D}_2} : \mathcal{D}_{2,*}^n \to \mathcal{D}_{2,*}$
be the $n$-ary functions corres\-ponding to the function symbol $f$
in $\mathcal{D}_{1}$ and $\mathcal{D}_{2}$, respectively.
We make use of the functions $\psi_*$ and $\psi$,
and obtain from \eqref{psicompatible} and the first condition in
Definition \ref{homisodef}(a) for all 
$t_1,\ldots,t_n  \in \hat{\mathcal{L}}_{1,*}$:
\begin{align*}
& (\psi \circ f_{\mathcal{D}_1})
(\mathcal{D}_1(t_1),\ldots,\mathcal{D}_1(t_n))\\
& = \psi ( \mathcal{D}_1( f(t_1 \ldots t_n)))\\
& = \mathcal{D}_2 (\psi_* ( f(t_1 \ldots t_n)))\\
& = \mathcal{D}_2 (f(\psi_*(t_1) \ldots \psi_*(t_n)))\\
& = f_{\mathcal{D}_2}(\mathcal{D}_2(\psi_*(t_1)),\ldots,
\mathcal{D}_2(\psi_*(t_n)))\\
& = f_{\mathcal{D}_2}(\psi(\mathcal{D}_1(t_1)),\ldots,
\psi(\mathcal{D}_1(t_n)))\,.
\end{align*}
Since $\mathcal{D}_1 : \hat{\mathcal{L}}_{1,*} 
\to \mathcal{D}_{1,*}$ is surjective, we have for all
$\vartheta_1,\ldots,\vartheta_n \in \mathcal{D}_{1,*}$:
\begin{align*}
 (\psi \circ f_{\mathcal{D}_1})(\vartheta_1,\ldots,\vartheta_n)=
 f_{\mathcal{D}_2}(\psi(\vartheta_1),\ldots,\psi(\vartheta_n))\,.
\end{align*}
This is a well known relation for the definition of an
homomorphism $\psi$ between the $n$-ary functions $f_{\mathcal{D}_1}$ 
and $f_{\mathcal{D}_2}$, see for example Hodges \cite[Chapter 1.2]{H},
and Definition \ref{homisodef} is indeed a generalization
of the corresponding definitions for the modeling of functions
and relations.
\end{rem}

\begin{exa}\label{klein}
The Cayley table for Klein's abelian four-group $V_4$ is given by
\begin{equation*}
		\begin{tabular}{c|cccc}
		~~ $\cdot$ ~~  & ~~ $e$ ~~  & ~~ $a$ ~~ & ~~ $b$ ~~  & ~~ $c$ ~~\\ \hline
		$e$  & \rule{0pt}{2.3ex}   $e$  &  $a$   &  $b$  & $c$ \\ 
		$a$  & \rule{0pt}{2.3ex}   $a$  &  $e$   &  $c$   & $b$ \\ 
		$b$  & \rule{0pt}{2.3ex}   $b$  &  $c$   &  $e$   & $a$ \\ 
		$c$  & \rule{0pt}{2.3ex}   $c$  &  $b$   &  $a$   & $e$ \\ 
		\end{tabular}
	\end{equation*}
	The group $V_4$ has the four distinct members $e,a,b,c$, 
	where $e$ is the identity element.
	
	We define the formal language $[M;\mathcal{L}]$  
	with $M = [A;P]$, $A=\{*\}$, $P = \emptyset$ and the set $\mathcal{L}$
	of terms generated by the following two rules:
	\begin{itemize}
	\item We have $x \in \mathcal{L}$ for any variable $x \in X$, 
	\item if $\lambda, \mu \in \mathcal{L}$ , then $*(\lambda \mu) \in \mathcal{L}$\,.
	\end{itemize}
	
	The symbols $e,a,b,c$ are not present in the formal language 
	$[M;\mathcal{L}]$, so we can use the set of names
	$\mathcal{N}=\{e,a,b,c\}$. 
	
	We obtain the formal language
	$[\hat{M};\hat{\mathcal{L}}]$ by adding the names in $\mathcal{N}$
	to $[M;\mathcal{L}]$. Let $\hat{\mathcal{L}}_*$ be 
	the set of all terms $\lambda \in \hat{\mathcal{L}}$ without variables.
	%, i.e. $\mbox{var}(\lambda)=\emptyset$. 
	We define $\mathcal{D}_*=\mathcal{N}$ and
	$\mathcal{D} : \hat{\mathcal{L}}_* \to \mathcal{D}_*$ 
	recursively as follows:
\begin{align*}
	\mathcal{D}(e)=e\,,\mathcal{D}(a)=a\,,\mathcal{D}(b)=b\,,
	\mathcal{D}(c)=c\\
	\mbox{and}~~\mathcal{D}(*(\lambda \mu))=\mathcal{D}(\lambda) \cdot 
	\mathcal{D}(\mu)
	~~\mbox{if}~~ \lambda, \mu \in \hat{\mathcal{L}}_*\,. 
\end{align*}
We do not have predicate symbols in $P$, so 
Condition (5) is vacuous, 
but we require for all $\lambda, \mu \in \hat{\mathcal{L}}_*$ that
$$\mathcal{D}(\sim \lambda,\mu)=\top \Leftrightarrow
\mathcal{D}(\lambda)=\mathcal{D}(\mu)\,.$$
Following Conditions (6),(7) we obtain a 
structure $\mathcal{D}$ of Klein's four-group
for the formal language $[M;\mathcal{L}]$.
Its universe is $\mathcal{D}_*=\{e,a,b,c\}$.
Let $\pi : \{a,b,c\} \to \{a,b,c\}$ be a bijective mapping, i.e. any permutation
of the three members $a,b,c$. Then 
\begin{equation*}
	\psi(g) =\begin{cases}
	e &\text{if}~ g=e,\\
	\pi(g) &\text{if}~g \in \{a,b,c\}\\
	\end{cases}
	\end{equation*}
defines an automorphism $\psi : \mathcal{D}_* \to \mathcal{D}_*$
of $\mathcal{D}$. 

The mapping $\psi_* : \hat{\mathcal{L}}_* \to \hat{\mathcal{L}}_*$
in \eqref{psidef} can also be calculated by 
\begin{align*}
	\psi_*(e)=e\,,\psi_*(a)=\pi(a)\,,\psi_*(b)=\pi(b)\,,\psi_*(c)=\pi(c)\\
	\mbox{and}~~\psi_*(*(\lambda \mu))=*(\psi_*(\lambda)\psi_*(\mu))
	~~\mbox{if}~~ \lambda, \mu \in \hat{\mathcal{L}}_*\,. 
\end{align*}
Note that $\psi_*$ is also a bijective mapping with
$\psi \circ \mathcal{D}=\mathcal{D} \circ \psi_*$.
%We see that $\psi_*$ reflects the properties of the automorphism $\psi$,
%but on the level of the formal language $[\hat{M},\hat{\mathcal{L}}]$
%rather than on the metamathe\-matical level. 

Finally we yet want to mention that the simple formal language $[M;\mathcal{L}]$ in our example admits many other structures beside groups and semigroups. In order to model only groups, we have to add basis
axioms for groups to the formal language $[M;\mathcal{L}]$, for example
the two axioms
\begin{align*}
G_1=\, \sim \, *(*(xy)z),*(x*(yz))\,,\\
G_2=\, \exists x \, \forall y ~\&~ \sim\,*(xy),y~\exists z 
\,\sim\,*(zy),x\,,
\end{align*}
where $x$, $y$, $z$ are three distinct given variables.
Here we may likewise replace the formula $G_1$ by 
the formula $\forall x \forall y \forall z\, 
\sim \, *(*(xy)z),*(x*(yz))$.
There results a formal mathematical system $[M_{Grp};\mathcal{L}]$
with $M_{Grp}=[\,\{*\};\emptyset;\{G_1;G_2\}\,]$, whose models are only groups.
\end{exa}

\begin{thm}\label{homthm1}
(a) If $\psi : \mathcal{D}_{1,*} \to \mathcal{D}_{2,*}$
is a homomorphism from $\mathcal{D}_{1}$ to $\mathcal{D}_{2}$
and $\varphi : \mathcal{D}_{2,*} \to \mathcal{D}_{3,*}$
a homomorphism from $\mathcal{D}_{2}$ to $\mathcal{D}_{3}$,
then the mapping 
$\varphi \circ \psi : \mathcal{D}_{1,*} \to \mathcal{D}_{3,*}$
is a homomorphism from $\mathcal{D}_{1}$ to $\mathcal{D}_{3}$.\\
(b) Let $\psi : \mathcal{D}_{1,*} \to \mathcal{D}_{2,*}$
be an isomorphism from $\mathcal{D}_{1}$ to $\mathcal{D}_{2}$
and let \mbox{$\varphi : \mathcal{D}_{2,*} \to \mathcal{D}_{3,*}$}
be an isomorphism from $\mathcal{D}_{2}$ to $\mathcal{D}_{3}$.
Then the composed mapping $\varphi \circ \psi : \mathcal{D}_{1,*} \to \mathcal{D}_{3,*}$
is an isomorphism from $\mathcal{D}_{1}$ to $\mathcal{D}_{3}$.\\
(c) If $\psi$ is an isomorphism from
$\mathcal{D}_{1}$ to $\mathcal{D}_{2}$, then 
$\psi^{-1} : \mathcal{D}_{2,*} \to \mathcal{D}_{1,*}$
is an isomorphism from $\mathcal{D}_{2}$ to $\mathcal{D}_{1}$.
\end{thm}
\begin{proof}
(a) We define the function 
$\varphi_* : \hat{\mathcal{L}}_{2,*} \to \hat{\mathcal{L}}_{3,*}$
as follows: We prescribe $\lambda' \in \hat{\mathcal{L}}_{2,*}$ and consider for $k \geq 0$ 
the complete list of distinct names $\beta_{d'_1},...,\beta_{d'_k}$ in $\lambda'$,
ordered according to their first occurrence in $\lambda'$. Let $x_1,...,x_k \in X$
be distinct variables and let $\tilde{\lambda} \in \mathcal{L}$ 
result from $\lambda'$
if we replace $\beta_{d'_1},...,\beta_{d'_k}$ everywhere in $\lambda'$ 
by $x_1,...,x_k$, respectively. Then we put
\begin{align}\label{phidef}
\varphi_*(\lambda')=\tilde{\lambda}\,\frac{\gamma_{\varphi(d'_1)}}{x_1}...
\frac{\gamma_{\varphi(d'_k)}}{x_k} \in \hat{\mathcal{L}}_{3,*}
\mbox{~~~for~~}
\lambda'=\tilde{\lambda}\,\frac{\beta_{d'_1}}{x_1}...
\frac{\beta_{d'_k}}{x_k} \in \hat{\mathcal{L}}_{2,*}\,.
\end{align}
We have seen in Remark \ref{rem_psidef} that formulas \eqref{psidef}
hold without further restrictions on variables $x_1,\ldots,x_m$
and names $\alpha_{d_1},\ldots,\alpha_{d_m} \in \mathcal{N}_1$.

Especially for any given $\lambda \in \hat{\mathcal{L}}_{1,*}$
we can find $\tilde{\lambda} \in \mathcal{L}$,
variables $x_1,...,x_m$ with 
$\mbox{var}(\tilde{\lambda}) \subseteq \{x_1,...,x_m\}$  
and $d_1,\cdots,d_m \in \mathcal{D}_{1,*}$
so that
\begin{align}\label{lambdastrich}
\lambda' = \psi_*(\lambda)
=\tilde{\lambda}\,\frac{\beta_{\psi(d_1)}}{x_1}...
\frac{\beta_{\psi(d_m)}}{x_m}  \in \hat{\mathcal{L}}_{2,*}
\mbox{~~~for~~}
\lambda=\tilde{\lambda}\,\frac{\alpha_{d_1}}{x_1}...
\frac{\alpha_{d_m}}{x_m}\,.
\end{align}
We can also follow the same reasoning from Remark \ref{rem_psidef} to calculate $\varphi_*(\lambda')$ directly from $\lambda'$
given in \eqref{lambdastrich}, instead of using \eqref{phidef}. 
We obtain
\begin{align*}
(\varphi_* \circ \psi_*)(\lambda)= \varphi_*(\lambda')=
\tilde{\lambda}\,\frac{\gamma_{\varphi(\psi(d_1))}}{x_1}...
\frac{\gamma_{\varphi(\psi(d_m)}}{x_m} \in \hat{\mathcal{L}}_{3,*}
\end{align*}
with $\tilde{\lambda} \in \mathcal{L}$ and 
$\begin{displaystyle}
\lambda=\tilde{\lambda}\,\frac{\alpha_{d_1}}{x_1}...
\frac{\alpha_{d_m}}{x_m} \in \hat{\mathcal{L}}_{1,*}\,.
\end{displaystyle}$

We see that $\varphi_* \circ \psi_* : 
\hat{\mathcal{L}}_{1,*} \to \hat{\mathcal{L}}_{3,*}$
is uniquely constructed from $\varphi \circ \psi$
in the same way as $\psi_*$ was obtained from $\psi$ and
$\varphi_*$ from $\varphi$, respectively,
i.e. $\varphi_* \circ \psi_* =(\varphi \circ \psi)_*$\,.
We have
\begin{align*}
\psi\left(\mathcal{D}_{1}(\lambda)\right)=
\mathcal{D}_{2}(\psi_*(\lambda))
\qquad \mbox{for~all~} \lambda \in \hat{\mathcal{L}}_{1,*} \,,
\end{align*}
as well as
\begin{align*}
\varphi\left(\mathcal{D}_{2}(\lambda')\right)=
\mathcal{D}_{3}(\varphi_*(\lambda'))
\qquad \mbox{for~all~} \lambda' \in \hat{\mathcal{L}}_{2,*} \,,
\end{align*}
and obtain with $\lambda' = \psi_*(\lambda) \in \hat{\mathcal{L}}_{2,*}$ that
\begin{align*}
(\varphi \circ \psi)(\mathcal{D}_1(\lambda))=
\varphi(\psi(\mathcal{D}_1(\lambda)))=
\varphi(\mathcal{D}_2(\psi_*(\lambda)))\\=
\mathcal{D}_3(\varphi_*(\psi_*(\lambda)))=
\mathcal{D}_3((\varphi_* \circ \psi_*)(\lambda))\,.
\end{align*}
We see that the first condition in Definition \ref{homisodef} 
required for $\varphi \circ \psi$ to be an homomorphism 
from $\mathcal{D}_1$
to $\mathcal{D}_3$ is satisfied.

We have $p_{0,1} \Rightarrow p_{0,2} \Rightarrow p_{0,3}$ and
\begin{align*}
(d_1,...,d_n) \in p_{n,1} \Rightarrow 
(\psi(d_1),...,\psi(d_n)) \in p_{n,2}\,, \\
(\psi(d_1),...,\psi(d_n)) \in p_{n,2} \Rightarrow 
((\varphi \circ \psi)(d_1),...,(\varphi \circ \psi)(d_n)) \in p_{n,3}
\end{align*}
for all $d_1,...,d_n \in \mathcal{D}_{1,*}$\,,
and the second condition in Definition \ref{homisodef} required for 
$\varphi \circ \psi$ to be an homomorphism from $\mathcal{D}_1$
to $\mathcal{D}_3$  is satisfied as well.\\
(b) This part is proven in the same way as part (a), by using 
the fact that the composition of two bijective functions is
again a bijective function.\\
(c) Now we assume that $\psi$ is an \textit{isomorphism}. Then we obtain 
from \eqref{psidef} that $\psi_*$ is bijective, hence
\begin{align*}
\mathcal{D}_{1}(\lambda))=
\psi^{-1}(\mathcal{D}_{2}(\psi_*(\lambda))
\qquad \mbox{for~all~} \lambda \in \mathcal{L}_{1,*} \,,\\
\mathcal{D}_{1}(\psi_*^{-1}(\lambda'))=
\psi^{-1}(\mathcal{D}_{2}(\lambda'))
\qquad \mbox{for~all~} \lambda' \in \mathcal{L}_{2,*} \,.
\end{align*}
Finally we have
\begin{align*}
p_{0,2} \Leftrightarrow p_{0,1} \mbox{~and~} (d'_1,...,d'_n) \in p_{n,2} \Leftrightarrow
 (\psi^{-1}(d'_1),...,\psi^{-1}(d'_n)) \in p_{n,1} 
\end{align*}
for all $p \in P$, for all $n \in \setN$ and for all 
$d'_1,...,d'_n \in \mathcal{D}_{2,*}$\,.
\end{proof}

\begin{defin}\label{henkin_model}
Let $[M;\mathcal{L}]$ with $M=[A;P]$ be a formal language
and let $\mathcal{D}$ be a structure for $[M;\mathcal{L}]$.
We define the formal mathematical system 
$[M(\mathcal{D});\mathcal{L}(D)]$ as follows:
We will first define the extension $[\hat{M};\hat{\mathcal{L}}]$
with $\hat{M}=[\hat{A};P]$ due to Condition (3) by adding the names of $\mathcal{D}$ to $[M;\mathcal{L}]$. Let $\hat{B}$
be the set of all formulas of $[\hat{M};\hat{\mathcal{L}}]$
which are valid in $\mathcal{D}$. Then we put
$\begin{displaystyle}
M(\mathcal{D})=[\hat{A};P;\hat{B}]\,, ~ 
\mathcal{L}(D)=\hat{\mathcal{L}}\,.
\end{displaystyle}$
We can also use the notation 
$[M(\mathcal{D});\mathcal{L}(D)]$ if $\mathcal{D}$
is a model of a formal mathematical system $[M;\mathcal{L}]$ 
with $M = [A;P;B]$, because we have $B \subseteq \hat{B}$.
\end{defin}

The formal mathematical language $[\hat{M};\hat{\mathcal{L}}]$ 
from Definition \ref{henkin_model}
has a counterpart in 
\cite[Chapter 2.5]{Shoenfield}, 
whereas the formal mathe\-ma\-tical system
$[M(\mathcal{D});\mathcal{L}(D)]$ 
has a counterpart in Shoenfields book
from \cite[Chapters 2.5, 5.5]{Shoenfield}. 
The reason for our definition is the following 

\begin{thm}\label{complete_henkin_model}
The formal mathematical system 
$[M(\mathcal{D});\mathcal{L}(D)]$ 
from Definition \ref{henkin_model}
is a complete Henkin system due to Definition \ref{henkin_def}, and hence it is also a consistent formal mathematical system.
\end{thm}
\begin{proof}
The theorem and its proof is given in
\cite[Section 4, Theorem 4.1]{Ku4}.
The proof given there does not make use of the countability
of the formal language $[M;\mathcal{L}]$ and hence holds in general.
\end{proof}
In general, even for a countably infinite 
language $[M;\mathcal{L}]$, we are not able
to determine all formulas in $[M;\mathcal{L}]$ 
which are valid in $\mathcal{D}$.
Therefore, Theorem \ref{complete_henkin_model}
is not a constructive statement. \\

Consider the formal system 
$[M';\mathcal{L}']$ from Definition \ref{extensions}
which extends a formal system $[M;\mathcal{L}]$.
The next definition and the following theorem provide a justification for calling $[M;\mathcal{L}]$ 
a \textit{generalization} of $[M';\mathcal{L}']$, by showing how any model of $[M';\mathcal{L}']$ can be restricted 
to a model of $[M;\mathcal{L}]$.

\begin{defin} \textit{Restriction of a structure}\\ 
\label{def_restriction}
Given are two formal mathematical systems $[M;\mathcal{L}]$ and 
$[M';\mathcal{L}']$ with $M=[A; P; B]$, $M'=[A'; P'; B']$ and
$A \subseteq A'$, 
$P \subseteq P'$,
$\mathcal{L} \subseteq \mathcal{L}'$\,.
Let $\mathcal{D}'$ be a structure for $[M';\mathcal{L}']$ 
with universe $\mathcal{D}_*$ and the set of names 
$\mathcal{N}=\{ \alpha_d : d \in \mathcal{D}_* \}\,$.
According to Condition (3) and \cite[Corollary (4.9)]{Ku}
we add the names
to $[M;\mathcal{L}]$ and $[M';\mathcal{L}']$ and obtain
\begin{align*}
[\hat{M'};\hat{\mathcal{L}'}]=
[M'_{A' \cup \mathcal{N}};\mathcal{L}'_{A' \cup \mathcal{N}}]\,,
[\hat{M};\hat{\mathcal{L}}]=
[M_{A \cup \mathcal{N}};\mathcal{L}_{A \cup \mathcal{N}}]\,.
\end{align*}
Then we obtain a \textit{structure} for
$[M;\mathcal{L}]$, the \textit{restriction} 
$\mathcal{D}=\mathcal{D}'|_{[M;\mathcal{L}]}$\,:
\begin{itemize}
\item[(i)] The mapping $\mathcal{D} : \hat{\mathcal{L}}_* \to \mathcal{D}_*$ is the restriction of
$\mathcal{D}' : \hat{\mathcal{L}'}_* \to \mathcal{D}_*$
to $\hat{\mathcal{L}}_*$. We see that $\mathcal{D}(\alpha_d)=d$
for all $d \in \mathcal{D}_*$, and $\mathcal{D}$ is surjective.
Now $\mathcal{D}$ satisfies Condition (4),
since $\hat{\mathcal{L}}=\mathcal{L}_{A \cup \mathcal{N}}$ 
is complete under substitutions due to \cite[Corollary (4.9)]{Ku}.
\item[(ii)] For $\mathcal{D}$ we only assign the predicates 
$p_n \subseteq \mathcal{D}_*^n$ to the predicate symbols $p \in P$ and to each $n \in \setN_0$.
\item[(iii)] We have an extension of the function $\mathcal{D}$ which assigns a truth value to each closed formula $F$ of $[\hat{M};\hat{\mathcal{L}}]$ due to $\mathcal{D}(F)=\mathcal{D}'(F)$.
Since every formula of $[\hat{M};\hat{\mathcal{L}}]$ is also a formula of $[\hat{M'};\hat{\mathcal{L}'}]$, this is a well defined extension.
\end{itemize}
\end{defin}

In this definition, $\mathcal{D}'$ and $\mathcal{D}$  both have the same universe $\mathcal{D}_*$ and the same set of names $\mathcal{N}$, naming the individuals in both structures in the same way.
Now we have an important generalization of
\cite[Chapter 4.2, Lemma 1]{Shoenfield}.

\begin{thm}\label{thm_restriction}
Given are two formal mathematical systems $[M;\mathcal{L}]$ and 
$[M';\mathcal{L}']$. Let $[M';\mathcal{L}']$ be an extension of $[M;\mathcal{L}]$ 
and let $\mathcal{D}'$ be a model for $[M';\mathcal{L}']$. Then the restriction
$\mathcal{D}=\mathcal{D}'|_{[M;\mathcal{L}]}$ is a model for $[M;\mathcal{L}]$.
\end{thm}
\begin{proof}
Due to Definition \ref{def_restriction} we already have the structure
$\mathcal{D}=\mathcal{D}'|_{[M;\mathcal{L}]}$ for $[M;\mathcal{L}]$.
Let $F$ be a basis axiom of $[M;\mathcal{L}]$. Then $F$ is provable
in $[M;\mathcal{L}]$ as well as in $[M';\mathcal{L}']$, since
$[M';\mathcal{L}']$ is an extension of $[M;\mathcal{L}]$.
Recall that both structures $\mathcal{D}'$, $\mathcal{D}$
have the same universe $\mathcal{D}_*$ and the same set of names 
$\mathcal{N}=\{\alpha_d\,:\,d \in \mathcal{D}_* \}$.
But $F$ is also provable in the extension
$[\hat{M'};\hat{\mathcal{L}'}]$ of $[M';\mathcal{L}']$.
We obtain from Lemma \ref{correct} that $F$ is valid 
in $\mathcal{D}'$, make use of Condition (7) and see from 
$\hat{\mathcal{L}}_* \subseteq \hat{\mathcal{L}'}_*$
that $F$ is also valid in $\mathcal{D}$.
Hence the restriction $\mathcal{D}$ 
is a model for $[M;\mathcal{L}]$. 
\end{proof}

\begin{rem}
If $[M;\mathcal{L}]$ or its extension $[M';\mathcal{L}']$
is a contradictory formal system, then we do not have a model
$\mathcal{D}'$ for $[M';\mathcal{L}']$.
\end{rem}

We have seen that we can extend the most important basic concepts of classical model theory to our new, generalized model theory.
The mathematical structures we have defined in our model theory can, in principle, all be obtained using set theory.
Axiomatic set theory provides some generally accepted rules for dealing with sets, and it intends to lay a foundation of mathematics by using only the primitive terms ``set" 
and ``membership".
More specific mathe\-matical structures are defined using these primitive terms.
Such a commonly accepted foundation of mathe\-ma\-tics is the Zermelo-Fraenkel set theory ZFC with the axiom of choice,
see Jech \cite{J} and Shoenfield \cite[Chapter 9]{Shoenfield}.
ZFC is only dealing with sets whose members are sets again.

In \cite{Kuset} we have presented a generalization of ZFC, 
starting with a fragment of axiomatic set theory
called RST, for reduced set theory.
As in ZFC we are only dealing with sets whose members are sets again.
However, unlike ZFC, RST is a simpler set theory that is better suited to our model theory. 

Next we will describe the basic principles of RST and its benefits:\\
\noindent
A set $U$ is called \textit{transitive} iff $Y \subseteq U$ 
for all $Y \in U$.\\
By $\mathcal{P}(Y) =\{ V : V \subseteq Y\}$
we denote the power set of $Y$.\\
Due to \cite[Definition 2.1]{Kuset} 
we say a set $U$ is \textit{subset-friendly} iff 
\begin{itemize}
\item[1.] $\emptyset \in U$\,.
\item[2.] $U$ is transitive\,.
\item[3.] For all $Y \in U$ we have $\mathcal{P}(Y) \in U$\,.
\item[4.] For all $Y, Z \in U$ we have a transitive set\\ 
$V \in U$ with $\{Y,Z\} \subseteq V$\,.
\end{itemize}
Now we are listing the six principles 
according to which we are dealing with sets in RST. 
For sets $A$, $B$, $U$, $V$, $Y$ these are given by
\begin{itemize}
\item[P1.] \textit{Principle of extensionality.}
If $A$ and $B$ have the same elements, then $A=B$.
\item[P2.] \textit{Subset principle.}
If $\mathcal{F}$ is a property which may depend on previously given sets, then we can form the subset of $A$ given by
$U = \{Y :\,\mbox{there~holds~} Y \in A 
\mbox{~and~} Y \mbox{~has~property~}\mathcal{F}\}\,.$\\
Especially the empty set $\emptyset$ can be obtained from this principle.
\item[P3.] \textit{Principle of regularity.}
If $U$ is not the empty set,
then we have $Y \in U$ with $U \cap Y=\emptyset$.
\item[P4.] \textit{Principle for pairing of sets.}
If $A$ and $B$ are given, then we can find a set $U$
with $A \in U$ and $B \in U$. 
We can combine this with (P2)
to form $U=\{A, B\}$.
\item[P5.] \textit{Principle for subset-friendly sets.}
If $A$ is given, then we have a subset-friendly set $U$
with $A \in U$.
\item[P6.] \textit{Principle of choice.}
If $U$ has only nonempty and pairwise disjoint elements
then we can find a set $Y$ with the following property:
For every member $A \in U$ there exists exactly one set $V$ with
$Y \cap A = \{V\}$.
\end{itemize}
The novel feature of (P5) is that it contains the set $A$ as 
para\-meter. Hence we can use it step by step.
We will first provide a subset-friendly set $U$ with $A=\emptyset \in U$.
Then we can apply (P5) to $A=U$ again, and so on.
The correctness of (P5) is guaranteed by \cite[Theorem 2.5]{Kuset}.
In this way we have a sufficiently large set as background available. 
Within this set we can perform the required set operations, and we apply the subset axioms directly instead of the replacement axioms.

A subset-friendly set $U$ satisfies the following closure properties:
\begin{itemize}
\item If $A \in U$, then $\cup(A) \in U$,
$\mathcal{TC}(A) \in U$ and $\mathcal{P}(A) \in U$,
\item If $A \subseteq V$ and $V \in U$, then $A \in U$,
\item If $A, B \in U$, then $A \cup B \in U$, $A \cap B \in U$
and $A \setminus B \in U$,
\item If $A_1, \ldots, A_n \in U$, then $\{ A_1, \ldots, A_n\} \in U$,
\item If $A_1, \ldots, A_n \in U$, 
then $A_1 \times \ldots \times A_n \in U$.
\end{itemize}
Here $\cup(A)$ denotes the union of a set $A$ and
$\mathcal{TC}(A)$ its transitive closure.
We have shown in \cite{Kuset} that all these closure properties of the
subset-friendly sets are formally provable in RST.

The new set theory RST now offers us the following advantages:
\begin{itemize}
\item[1.] RST is designed for pure and well-founded sets according to (P1) and (P3). It follows 
from \cite[Theorem 5.1]{Kuset} that RST admits tran\-si\-tive 
well-founded models which can be extended by adding step by step appropriate new axioms to RST. Then the former transitive model just becomes a transitive set and a member of the extended model. In this way we can extend RST and its transitive models, whenever this is needed.
\item[2.] On the other hand, we have seen in \cite{Kuset} that even the simplest
models of RST are large and rich enough in order to formalize
most parts of classical mathematics. Hence we can also study
axiomatic set theory within the theory of models, 
using universal sets instead of proper classes.
\item[3.] Axiomatic set theory cannot provide an absolutely valid description for large and especially uncountable ordinal numbers,
see \cite[Theorem 5.1, Theorem 5.3]{Ku4}. 
By using the reduced set theory RST, we do not make existence statements for such ordinals.
In \cite[Theorem 5.1]{Kuset} we obtained transitive well-founded models for RST whose ordinals are countable and definable in an absolute sense, i.e. they have the same meaning in all these models.
\item[4.] The subset-friendly theories in 
\cite[Definition 4.2]{Kuset} are extensions of RST with new axioms and new symbols in the formulas. They are well suited for formulating general axiomatic theories in which set theo\-ry is available in the background. 
A subset-friendly theory is an extension of RST 
in the sense of Definition \ref{extensions} such that the subset axioms that correspond to (P2) remain valid,
but with new symbols in the formulas.

In the paper \cite{Kuset}, we also presented some 
fundamental results to obtain conservative extensions 
of RST according to Definition \ref{extensions}(b). 
\end{itemize}

In our final example we briefly touch on category theory.
We make use of Tom Leinster's textbook \cite[Section 1]{TL}:
\begin{exa}\label{cat} 
The formal system CAT for categories
\begin{itemize}
\item[1.] We put $A=\{1,d,c\}$ and 
$P=\{\mbox{Ob}, \mbox{Ar}\}$.
\item[2.] Let $\mathcal{L}$ be the set of lists generated by the following rules:
\begin{itemize}
\item[$\bullet$] $z \in \mathcal{L}$ for all variables $z \in X$.
\item[$\bullet$]  For all $\lambda \in \mathcal{L}$ we have
$1(\lambda) \in \mathcal{L}$,
$d(\lambda) \in \mathcal{L}$ and
$c(\lambda) \in \mathcal{L}$.
\item[$\bullet$]  For all $\lambda, \mu \in \mathcal{L}$ we have
$\lambda \mu \in \mathcal{L}$. 
\end{itemize}
\item[3.] Let $x,u,v \in X$ be three distinct variables.
\end{itemize}

We define the formal mathematical system 
$\mbox{CAT}=[M;\mathcal{L}]$
with\\ 
$M=[A;P;B]$ and $B$ consisting on the following ten\\
Basis Axioms for $\mbox{CAT}$:
\begin{itemize}
\item[C1.] $\begin{displaystyle}
\to ~ \mbox{Ar}\,v ~\, \mbox{Ob}\,d(v)
\end{displaystyle}$
\item[C2.] $\begin{displaystyle}
\to ~ \mbox{Ar}\,u ~\, \mbox{Ob}\,c(u)
\end{displaystyle}$
\item[C3.] $\begin{displaystyle}
\to ~ \mbox{Ar}\,v ~ \to ~ \mbox{Ar}\, u
\to ~ \sim \, d(v),c(u) ~\, \mbox{Ar}\, vu
\end{displaystyle}$
\item[C4.] $\begin{displaystyle}
\to ~ \mbox{Ar}\,v ~ \to ~ \mbox{Ar}\, u
\to ~ \sim \, d(v),c(u) ~ \sim\, d(vu), d(u)
\end{displaystyle}$
\item[C5.] $\begin{displaystyle}
\to ~ \mbox{Ar}\,v ~ \to ~ \mbox{Ar}\, u
\to ~ \sim \, d(v),c(u) ~ \sim\, c(vu), c(v)
\end{displaystyle}$
\item[C6.] $\begin{displaystyle}
\to ~ \mbox{Ob}\,x ~\, \mbox{Ar}\,1(x)
\end{displaystyle}$
\item[C7.] $\begin{displaystyle}
\to ~ \mbox{Ob}\,x ~ \sim\,d(1(x)),x
\end{displaystyle}$
\item[C8.] $\begin{displaystyle}
\to ~ \mbox{Ob}\,x ~ \sim\,c(1(x)),x
\end{displaystyle}$
\item[C9.] $\begin{displaystyle}
\to ~ \mbox{Ar}\,v ~ \sim\,v 1(d(v)),v
\end{displaystyle}$
\item[C10.] $\begin{displaystyle}
\to ~ \mbox{Ar}\,u ~ \sim\,1(c(u))u,u
\end{displaystyle}$

\end{itemize}
Let $\mathcal{D}$ be a model of $\mbox{CAT}$ with universe 
$\mathcal{D}_*$ and the set of names
$\mathcal{N}=\{\alpha_a\,:\,a \in \mathcal{D}_*\}$\,.
We make use of Condition (5)
and have $\mbox{Ob}_1 \subseteq \mathcal{D}_*$,
$\mbox{Ar}_1 \subseteq \mathcal{D}_*$.
Since $\mbox{Ob}$ and $\mbox{Ar}$
only occur as 1-ary predicate symbols in the basis axioms of CAT,
we can use \cite[Section 3.4]{Ku2} and may assume 
for the model $\mathcal{D}$ that
\begin{align*}
\mbox{Ob}_0=\mbox{Ar}_0=\emptyset\,,\quad
\mbox{Ob}_n=\mbox{Ar}_n=\emptyset \quad \mbox{for~all~} 
n \geq 2\,.
\end{align*}
Let $U$ be a subset-friendly set 
with $\mathcal{D}_* \in U$.
In order to obtain a category $\mathscr{A}$ 
from $\mathcal{D}$ in $U$ we proceed as follows:
We define the set of objects
$\begin{displaystyle}
ob(\mathscr{A})=\{a \in \mathcal{D}_*\,:\, 
\mathcal{D}(\mbox{Ob}\,\alpha_a)=\top\,\}=\mbox{Ob}_1\,.
\end{displaystyle}$\\
For any two $a,b \in ob(\mathscr{A})$ we also have a set of arrows
\begin{align*}
\mathscr{A}(a,b)=\{f \in \mathcal{D}_* :  
\mathcal{D}(\mbox{Ar}\,\alpha_f)=
\mathcal{D}(\sim d(\alpha_f),\alpha_a)=
\mathcal{D}(\sim c(\alpha_f),\alpha_b)=\top\,\}\,.
\end{align*}
Note that $ob(\mathscr{A}) \in U$ and
$\mathscr{A}(a,b) \in U$, since 
$\mathscr{A}(a,b) \subseteq \mbox{Ar}_1 \subseteq \mathcal{D}_*$.
For each three objects $a,b,\tilde{c} \in ob(\mathscr{A})$ we have an own composition function
\begin{align*}
\circ : \mathscr{A}(b,\tilde{c}) \times \mathscr{A}(a,b) \to 
\mathscr{A}(a,\tilde{c}) ~\mbox{with}~
\circ(g,f) = \mathcal{D}(\alpha_g \alpha_f)\,.
\end{align*}
For this reason we may shortly write $gf$ instead of $\circ(g,f)$.
The composition function is empty if and only if $\mathscr{A}(a,b)=\emptyset$
or $\mathscr{A}(b,\tilde{c})=\emptyset$.
Finally, for each $a \in ob(\mathscr{A})$ we have an own identity on $a$, namely
\begin{align*}
1_a = \mathcal{D}(1(\alpha_a)) \in \mathscr{A}(a,a) \,.
\end{align*}
We conclude our example with the following remarks:
\begin{itemize}
\item[A.] It may happen that $ob(\mathscr{A})=\emptyset$.
Then we obtain the so called empty category.
But the formula $\exists x \,\sim\,x,x$ 
is provable in \mbox{CAT}.
Then $\mathcal{D}(\exists x \,\sim\,x,x)=\top$
from Lemma \ref{correct} and hence $\mathcal{D}_* \neq \emptyset$.
\item[B.] Assume that $a,a',b,b' \in ob(\mathscr{A})$
and that $f \in \mathscr{A}(a,b) \cap \mathscr{A}(a',b')$.
Then we obtain from Condition (6.1) 
and Lemma \ref{structure_lemma}(a):
\begin{align*}
\mathcal{D}(\sim d(\alpha_f),\alpha_a)=
\mathcal{D}(\sim d(\alpha_f),\alpha_{a'})=\top\,,\\
\mathcal{D}(d(\alpha_f))=\mathcal{D}(\alpha_a)=
\mathcal{D}(\alpha_{a'})=a=a'\,,
\end{align*}
and similarly
\begin{align*}
\mathcal{D}(\sim c(\alpha_f),\alpha_b)=
\mathcal{D}(\sim c(\alpha_f),\alpha_{b'})=\top\,,\\
\mathcal{D}(c(\alpha_f))=\mathcal{D}(\alpha_b)=
\mathcal{D}(\alpha_{b'})=b=b'\,.
\end{align*}
\item[C.] For $\mbox{CAT}=[M;\mathcal{L}]$ we obtain the formal system $[\hat{M};\hat{\mathcal{L}}]$ by adding the names in 
$\mathcal{N}$ to $[M;\mathcal{L}]$, see Condition (3). 
Let $\hat{\mathcal{L}}_*$ be the set of all terms 
$\lambda \in \hat{\mathcal{L}}$ without variables.
We make use of the notion $\mbox{SbL}$
for substitution in lists from \cite[(3.5)]{Ku}
and we obtain for all $\lambda, \mu \in \hat{\mathcal{L}}_*$
from Condition (4) with two variables $x,y \in X$:
\begin{align*}
\mathcal{D}(\lambda \mu)=\mathcal{D}(\mbox{SbL}(x \mu;\lambda;x))
=\mathcal{D}(\alpha_{\mathcal{D}(\lambda)}\mu)\\
=\mathcal{D}(\mbox{SbL}(\lambda y;\mu;y))
=\mathcal{D}(\lambda \alpha_{\mathcal{D}(\mu)})\\
=\mathcal{D}(\mbox{SbL}(\alpha_{\mathcal{D}(\lambda)}y;\mu;y))
=\mathcal{D}(\alpha_{\mathcal{D}(\lambda)}\alpha_{\mathcal{D}(\mu)})=\mathcal{D}(\alpha_{\mathcal{D}(\lambda\mu)})\,.\\
\end{align*}
\item[D.] We make use of Remark C and Lemmas \ref{structure_lemma},
\ref{correct} and obtain from \mbox{CAT} 
that $h(gf)=(hg)f$ 
for all $a_1,a_2,a_3,a_4 \in ob(\mathscr{A})$, for all
$f \in \mathscr{A}(a_1,a_2)$, $g \in \mathscr{A}(a_2,a_3)$,
$h \in \mathscr{A}(a_3,a_4)$ (associativity law) as well as
$r 1_a=r=1_b r$ for all $a,b \in ob(\mathscr{A})$
and for all $r$ with $r \in \mathscr{A}(a,b)$ (identity laws).
Hence $\mathscr{A}$ is indeed a category 
in a subset-friendly set $U$. 
In this way, any model $\mathcal{D}$ for \mbox{CAT}
in $U$ represents a category $\mathscr{A}$ in $U$.

We will not require any further properties for $U$. 
In particular, $U$ need not be a Grothendieck universe, which brings into play the more general concept of subset-friendly sets.

\item[E.] 
Using subset-friendly sets $U \in U'$ from (P5) and a formal system including basis \mbox{axioms}, for example 
an extension $[M';\mathcal{L}']$ of \mbox{CAT}, 
we consider \textit{all} structure-preserving 
maps $\psi : a \to b$ from Theorem \ref{homthm1} 
between the \textit{models} of the given formal system 
with $a, b, \psi \in U$. These are the arrows
of a special category $\mathscr{A}_{|}'$ in $U'$, and the domain $a$ 
as well as the codomain $b$ of each such function $\psi$
serve as the objects of $\mathscr{A}_{|}'$. 
The prescribed formal system may be contradictory,
in which case $\mathscr{A}_{|}'$ is the empty category from Remark A.

We note that all basis axioms of $\mbox{CAT}$ have the 
special form \mbox{$\to F_1 \cdots \to F_{n-1} F_n$} with $n \geq 2$
and prime formulas $F_1, \cdots, F_n$.
This fact can be used in order to obtain
a model $\mathcal{D}$ for \mbox{CAT}
in $U'$ which represents the category $\mathscr{A}_{|}'$, but
we will not go into tech\-nical details. 
Instead we will mention some useful properties which can be incorporated into 
the definition of the model $\mathcal{D}$.
Let $[M(\mathcal{D});\mathcal{L}(D)]$ 
be the Henkin-system from Definition \ref{henkin_model}
for $[M;\mathcal{L}]=\mbox{CAT}$ and let $\mathcal{N}$ be the set of names for $\mathcal{D}$. Then we have a parti\-cular
name $| \in \mathcal{N}$ with $\mathcal{D}(\,|\,)=U \in U'$
and $\mathcal{D}(\kappa) \in U$ for all $\kappa \in \mathcal{N} \setminus \{\,|\,\}$ which satisfies following properties: 
If we have a variable-free list $\lambda \in \mathcal{L}(D)_*$ 
such that $\mathcal{D}(\lambda)$ is not defined in the category 
$\mathscr{A}_{|}'$, then we have
$\mathcal{D}(\lambda)=\mathcal{D}(\,|\,)=U$\,.
Here $\mathcal{D}(\lambda)$ 
is defined in $\mathscr{A}_{|}'$ iff it is 
an arrow or object in $\mathscr{A}_{|}'$.
If $\lambda$ contains the symbol $\,|\,$,
then it is required that $\mathcal{D}(\lambda)$ 
is not defined in $\mathscr{A}_{|}'$. 
Especially, if we have any two lists 
$\lambda, \mu \in \mathcal{L}(D)_*$ such that 
$\circ (\mathcal{D}(\lambda),\mathcal{D}(\mu))$
is not defined in $\mathscr{A}_{|}'$, then we must have
$\mathcal{D}(\lambda \mu)=\mathcal{D}(\,|\,)=U$\,.
Finally, there must hold $\mathcal{D}(F)=\bot$ 
for any variable-free prime formula 
$F=\mbox{Ob} \, \lambda$ or 
$F=\mbox{Ar} \, \lambda$ of 
$[M(\mathcal{D});\mathcal{L}(D)]$ 
which contains the symbol $\,|\,$. 
\item[F.] In order to formally define special types of categories 
we will now consider an extension $[M';\mathcal{L}']$ of \mbox{CAT}.
Let $\psi' : \mathcal{D}_{1,*} \to \mathcal{D}_{2,*}$ 
be a structure-preserving map between models 
$\mathcal{D}'_{1}$ and $\mathcal{D}'_{2}$
of $[M';\mathcal{L}']$ with universes 
$\mathcal{D}_{1,*}$ and $\mathcal{D}_{2,*}$
and corresponding set of names 
$\mathcal{N}_1=\{\alpha_{\tilde{d}}\,:
\,\tilde{d} \in \mathcal{D}_{1,*}\}\,,
\mathcal{N}_2=\{\beta_{\tilde{d}'}\,:
\,\tilde{d}' \in \mathcal{D}_{2,*}\}\,.$
As in the previous remark we assume that 
$\psi', \mathcal{D}_{1,*},\mathcal{D}_{2,*} \in U$
with $U \in U'$ and two prescribed subset-friendly sets $U$, $U'$.\\
Let $\mathcal{D}_1=\mathcal{D}'_{1}|_{[M;\mathcal{L}]}$ and 
$\mathcal{D}_2=\mathcal{D}'_{2}|_{[M;\mathcal{L}]}$ 
with $[M;\mathcal{L}]=\mbox{CAT}$ be the restrictions to models for $\mbox{CAT}$ from Theorem \ref{thm_restriction}. 
Let $[\hat{M}_1;\hat{\mathcal{L}_1}]$ and
$[\hat{M}_2;\hat{\mathcal{L}_2}]$ result from
$\mbox{CAT}$ by adding the names in $\mathcal{N}_1$
and $\mathcal{N}_2$ to $\mbox{CAT}$, respectively.
Then $\psi=\psi'$ is also a structure-preserving map between 
$\mathcal{D}_{1}$ and $\mathcal{D}_{2}$, with a corresponding
map $\psi_* : \hat{\mathcal{L}}_{1,*} \to \hat{\mathcal{L}}_{2,*}$
which is only defined for the variable-free lists 
of $\hat{\mathcal{L}}_{1}$.
Let $\mathscr{A}$ be the category represented by
$\mathcal{D}_{1}$ in $U$ and let $\mathscr{B}$ be the category represented by $\mathcal{D}_{2}$ in $U$. 
Let $Ob_{1,1}=ob(\mathscr{A})$ be the set of objects 
of $\mathscr{A}$ and let $Ob_{1,2}=ob(\mathscr{B})$ 
be the set of objects of $\mathscr{B}$.
Let $Ar_{1,1}$ be the set of arrows
of $\mathscr{A}$ and let $Ar_{1,2}$ 
be the set of arrows of $\mathscr{B}$.
Then we have $Ob_{1,1}, Ob_{1,2}, Ar_{1,1}, Ar_{1,2} \in U$,
and $F=\psi|_{Ob_{1,1} \cup Ar_{1,1}}$
gives a well-defined \textit{functor} 
$F : \mathscr{A} \to \mathscr{B}$
in the sense of \cite[Section 1.2]{TL}.
In our example we can also write $F$ as a ``usual map"
$F : Ob_{1,1} \cup Ar_{1,1} \to Ob_{1,2} \cup Ar_{1,2}$.
\textit{All} these functions $F$ give the arrows
of a category $\mathscr{A}'$ in $U'$, whereas domain 
and codomain of each such function
serve as the objects of $\mathscr{A}'$. 
Note that $Ob_{1,1} \cap Ar_{1,1}=\emptyset$ is not required.
\item[G.] Interestingly, the basis axioms of $\mbox{CAT}$
can also be used to construct a recursive system: Starting with the ten basis axioms of $\mbox{CAT}$ as basis $R$-axioms and with finite extensions of $A$ and $P$, we may add finitely many appropriate basis $R$-axioms \mbox{according} 
to \cite[Section 1]{Ku}. In this way we can also obtain 
models of $\mbox{CAT}$ from \cite[Theorem 3.5, Theorem 3.6]{Ku2}. 
\end{itemize}
\end{exa}

The axiomatic method can help us eliminate erroneous  assumptions about mathematical structures that manifest as contradictions within formal mathematical systems. Moreover, a sound axio\-matic framework is necessary to make the underlying principles of a theory transparent to all its users.

But where do we get mathematical structures and the theories that describe their most important properties?
The driving force behind the development of mathematical theories is described in a philosophical attitude called Plato\-nism. Here we quote from the introduction to Skolem's lecture notes \cite[Section 1]{Skolem}:\\

%%%%%%%%%%%%%%%%%%%%%%%%%%%%%%%%%%
``There is one fact to which I would like to call attention.  Most of mathematics and perhaps above all the classical set theory has been developed in accordance with the philosophical attitude called Plato\-nism. This standpoint means that we consider the mathematical 
objects as existing before and independent of our actual thinking. Perhaps an illustrating way of expressing it is to say that when we are thinking about mathematical objects we are looking at eternal preexisting objects.
It seems clear that the word ``existence'' according to Platonism must have an absolute meaning so that everything we talk about shall either exist or not in a definite way. This is the philosophical background for classical mathematics generally and perhaps in particular for classical set theory. Being aware of this, Cantor explicitly cites Plato. Everybody is used to saying that a mathematical fact has been discovered, not that it has been invented. That shows our natural tendency towards Platonism.'' 

For a further discussion see also Putnam \cite{Putnam}.

\end{document}